\documentclass{article}
\usepackage{cite}
\usepackage{amsmath,amssymb,amsthm,mathrsfs}
\usepackage{titlesec,hyperref}
\usepackage{color}
\usepackage{fancyhdr}
\usepackage{cases}
\usepackage[margin=2.5cm]{geometry}
\usepackage{enumerate}
\usepackage[integrals]{wasysym}
\usepackage{bm}
\usepackage{graphicx}
\usepackage{caption}
\usepackage{float}
\usepackage{subfigure}

\newtheorem{theo}{Theorem}[section]
\newtheorem{lemm}{Lemma}[section]

\newtheorem{prop}{Proposition}[section]
\newtheorem{rema}{Remark}[section]
\numberwithin{equation}{section}

\def\ud{{\rm d}}

\allowdisplaybreaks

\begin{document}

\title{Ill-posedness for the Cauchy problem of the Camassa-Holm equation in $B^{1}_{\infty,1}(\mathbb{R})$}
\author{
Yingying $\mbox{Guo}^{1}$ \footnote{email: guoyy35@fosu.edu.cn}
\quad
Weikui $\mbox{Ye}^{2}$ \footnote{email: 904817751@qq.com}
\quad and \quad
Zhaoyang $\mbox{Yin}^{3}$ \footnote{email: mcsyzy@mail.sysu.edu.cn}\\
$^1\mbox{School}$ of Mathematics and Big Data, Foshan University,\\
Foshan, 528000, China\\
$^2\mbox{Institute}$ of Applied Physics and Computational Mathematics,\\
P.O. Box 8009, Beijing 100088, P. R. China\\
$^3\mbox{Department}$ of Mathematics, Sun Yat-sen University,\\
Guangzhou, 510275, China\\
$^3\mbox{Faculty}$ of Information Technology,\\
Macau University of Science and Technology, Macau, China
}
\date{}
\maketitle
\begin{abstract}
For the famous Camassa-Holm equation, the well-posedness in $B^{1+\frac{1}{p}}_{p,1}(\mathbb{R})$ with $ p\in [1,\infty)$ and the ill-posedness in $B^{1+\frac{1}{p}}_{p,r}(\mathbb{R})$ with $ p\in [1,\infty],\ r\in (1,\infty]$ had been studied in \cite{d1,d2,glmy,yyg}, that is to say, it only left an open problem in the critical case $B^{1}_{\infty,1}(\mathbb{R})$ proposed by Danchin in \cite{d1,d2}. In this paper, we solve this problem by proving the norm inflation and hence the ill-posedness for the Camassa-Holm equation in $B^{1}_{\infty,1}(\mathbb{R})$. Therefore, the well-posedness and ill-posedness for the Camassa-Holm equation in all critial Besov spaces $B^{1+\frac{1}{p}}_{p,1}(\mathbb{R})$ with $ p\in [1,\infty]$ have been completed. Finally, since the norm inflation occurs by choosing an special initial data $u_0\in B^{1}_{\infty,1}(\mathbb{R})$ but $u^2_{0x}\notin B^{0}_{\infty,1}(\mathbb{R})$ (an example implies $B^{0}_{\infty,1}(\mathbb{R})$ is not a Banach algebra), we then prove that this condition is necessary. That is, if $u^2_{0x}\in B^{0}_{\infty,1}(\mathbb{R})$ holds, then the Camassa-Holm equation has a unique solution $u(t,x)\in \mathcal{C}_T(B^{1}_{\infty,1}(\mathbb{R}))\cap \mathcal{C}^{1}_T(B^{0}_{\infty,1}(\mathbb{R}))$ and the norm inflation will not occur.
\end{abstract}


\section{Introduction}
\par
In the paper, we consider the following Cauchy problem for the Camassa-Holm (CH) equation \cite{ch}
\begin{equation}\label{ch}
\left\{\begin{array}{l}
u_t-u_{xxt}+3uu_{x}=2u_{x}u_{xx}+uu_{xxx},\qquad t>0,\ x\in\mathbb{R},\\
u(0,x)=u_0(x),\qquad x\in\mathbb{R}.
\end{array}\right.
\end{equation}
The CH equation is completely integrable \cite{c3,cgi} and admits infinitely many conserved quantities, an infinite hierarchy of quasi-local symmetries, a Lax pair and a bi-Hamiltonian structure \cite{c1,or}.

There are a lot of literatures devoted to studying the CH equation. Li and Olver \cite{lio} (see also \cite{rb}) established the local well-posedness for the Cauchy problem \eqref{ch} in Sobolev space $H^{s}(\mathbb{R})(s>\frac 3 2)$. Byers \cite{b} obtained that the Cauchy problem \eqref{ch} is ill-posed in $H^{s}(\mathbb{R})$ for $s<\frac{3}{2}$ in the sense of norm inflation. Then, by the Littlewood-Paley decomposition theory and transport equations theory, Danchin \cite{d1,d2}, Li and Yin \cite{liy} et al proved that the Cauchy problem \eqref{ch} is well-posed in Besov spaces $B^{s}_{p,r}(\mathbb{R})$ with $s>\max\{\frac{3}{2},1+\frac{1}{p}\},\ r<+\infty$ or $s=1+\frac{1}{p},\ p\in[1,2],\ r=1$ ($1+\frac{1}{p}\geq\frac{3}{2}$). Li et al. \cite{lyz1} demonstrated the non-continuity of the CH equation in $B^{\sigma}_{p,\infty}(\mathbb{R})$ with $\sigma>2+\max\{\frac32,1+\frac1p\}$ by constructing a initial data $u_{0}$ such that corresponding solution to the CH equation that starts from $u_{0}$ does not converge back to $u_{0}$ in the norm of $B^{\sigma}_{p,\infty}(\mathbb{R})$ as time goes to zero. Recently, Guo et al. \cite{glmy} established the ill-posedness for the Camassa-Holm type equations in Besov spaces $B^{1+\frac 1 p}_{p,r}(\mathbb{R})$ with $p\in[1,+\infty],\ r\in(1,+\infty]$, which implies $B^{1+\frac 1 p}_{p,1}(\mathbb{R})$ is the critical Besov space for the CH equation. However, there is no answer to the space $B^{1+\frac{1}{p}}_{p,1}(\mathbb{R})$ with $p\in(2,+\infty]$ ($1+\frac{1}{p}<\frac{3}{2}$). The main difficult is that the CH equation induce a loss of one order derivative in the stability estimates. To overcome this difficult, we adopt the compactness argument and Lagrangian coordinate transformation in our upcoming article \cite{yyg} rather than the usual techniques used in \cite{liy} to obtain the local well-posedness for the Cauchy problem of the Camassa-Holm type equations in $B^{1+\frac 1 p}_{p,1}(\mathbb{R})$ with $p\in[1,+\infty)$. This implies $B^{1+\frac 1 p}_{p,1}(\mathbb{R})$ is the critical Besov space and the index $\frac 3 2$ is not necessary for the Camassa-Holm type equations. 

Furthermore, in contrast to the KdV equation who can not describe the wave breaking phenomena observed in nature, the CH equation not only possesses global strong solutions \cite{ce2,ce4,ce5}, but also has wave-breaking phenomena \cite{c2,ce3,lio}. When this happens, the solution remains H\"{o}lder continuous and uniformly bounded, but develops an unbounded slope in finite time \cite{c2}. The questions about how to continue the solution beyond wave breaking can be nicely studied in the case of multipeakons. Multipeakons are given by (see \cite{ch} and references therein)
\begin{align}
u(t,x)=\sum\limits_{i=1}^{N}p_i(t)e^{-|x-q_i(t)|}\label{mult}
\end{align}
where $p_i,\ q_i$ satisfy
\begin{equation*}
\left\{\begin{array}{ll}
\frac{{\ud}p_i}{{\ud}t}=\sum\limits_{i\neq j}^{N}p_ip_j{\rm sign}(q_i-q_j)e^{-|q_i-q_j|},\\
\frac{{\ud}q_i}{{\ud}t}=\sum\limits_{j}^{N}p_je^{-|q_i-q_j|}.
\end{array}\right.
\end{equation*}
Observe that the solution \eqref{mult} is not smooth even with continuous functions $(p_{i}(t),q_{i}(t))$, one possible way to interpret \eqref{mult} as a weak solution of \eqref{ch} is to rewrite \eqref{ch} as
\begin{equation}\label{nonch}
\left\{\begin{array}{l}
u_{t}+uu_{x}+\partial_{x}(1-\partial_{xx})^{-1}\Big(u^{2}+\frac{u^{2}_{x}}{2}\Big)=0,\qquad t>0,\ x\in\mathbb{R},\\
u(0,x)=u_0(x),\qquad x\in\mathbb{R}.
\end{array}\right.
\end{equation}\\

First, we recall some recent results for the CH equation in critical Besov spaces.
\begin{theo}[See \cite{yyg}]\label{wellch}
Let $u_0\in B^{1+\frac{1}{p}}_{p,1}(\mathbb{R})$ with $1\leq p<\infty$ (~when $p=\infty,$ we choose $u_0\in B^{1+\epsilon}_{\infty,1}(\mathbb{R}),\ \forall \epsilon>0$). Then there exists a time $T>0$ such that the CH equation with the initial data $u_{0}$ is locally well-posed in the sense of Hadamard.
\end{theo}
\begin{theo}[See \cite{glmy}]\label{illch}
Let $1\leq p\leq\infty,\ 1<r\leq\infty$. For any $\varepsilon>0$, there exists  $u_0\in H^{\infty}(\mathbb{R})$ such that the following hold:
\begin{itemize}
\item [{\rm (1)}] $\|u_{0}\|_{B^{1+\frac{1}{p}}_{p,r}}\leq\varepsilon;$
\item [{\rm (2)}] There is a unique solution $u\in \mathcal{C}_{T}\big(H^{\infty}(\mathbb{R})\big)$ to the Cauchy problem \eqref{ch} with $T<\varepsilon;$
\item [{\rm (3)}] $\limsup\limits_{t\rightarrow T^{-}}\|u(t)\|_{B^{1+\frac{1}{p}}_{p,r}}\geq\limsup\limits_{t\rightarrow T^{-}}\|u(t)\|_{B^{1}_{\infty,\infty}}=\infty.$
\end{itemize}
\end{theo}
\begin{theo}[See \cite{lps}]\label{c01ill}
Given $u_{0}\in X\equiv H^{1}(\mathbb{R})\cap W^{1,\infty}(\mathbb{R})$, there exist
a nonincreasing function $T>0$ and a unique solution
$u=u(x, t)$ to \eqref{ch} such that
$u\in Z_{T}\equiv \mathcal{C}\big([-T,T];H^{1}(\mathbb{R})\big)\cap L^{\infty}\big([-T,T];W^{1,\infty}(\mathbb{R})\big)\cap \mathcal{C}^{1}\big((-T,T);L^{2}(\mathbb{R})\big)$
with
$\sup\limits_{[-T,T]}\|u(t)\|_{X}
=\sup\limits_{[-T,T]}\big(\|u(t)\|_{H^{1}}+\|u(t)\|_{W^{1,\infty}}\big)\leq c\|u_{0}\|_{X}$
for some universal constant $c>0$. Moreover, given $B>0$, the map
$u_{0}\mapsto u$, taking the data to the solution, is continuous from the ball
$\{u_{0}\in X: \|u_{0}\|_{X}\leq B\}$ into $Z_{T(B)}$.
\end{theo}
\begin{rema}
Indeed, we see from Theorem \ref{c01ill} that the continuous dependence does not hold in $L^{\infty}_T(\mathcal{C}^{0,1}(\mathbb{R}))$. The well-known peakon solution $ce^{-|x-ct|}$ is a counter example (see \cite{lps}). This implies the CH equation is ill-posed in $L^{\infty}_T(\mathcal{C}^{0,1}(\mathbb{R}))$. However, this time the norm inflation can not be appeared in $L^{\infty}_T(\mathcal{C}^{0,1}(\mathbb{R}))$. Indeed, for any $u_0\in \mathcal{C}^{0,1}$, one can find a lifespan $T\approx \frac{1}{\|u_0\|_{\mathcal{C}^{0,1}}}$ and a local unique solution $u(t,x)$ such that $\|u\|_{L^{\infty}_{T}(\mathcal{C}^{0,1})}\leq C\|u_0\|_{\mathcal{C}^{0,1}}$. This means the norm inflation can not be appeared in $L^{\infty}_{T}(\mathcal{C}^{0,1}(\mathbb{R}))$, which implies the norm inflation is just a sufficient but unnecessary conditions for the local ill-posedness.
\end{rema}
From Theorems \ref{wellch} and \ref{illch}, we have known that the local well-posedness and ill-posedness for the Cauchy problem \eqref{ch} of the CH equation have been completed in all critical Besov spaces except for $B^1_{\infty,1}(\mathbb{R})$. Note that the CH equation is the high-frequency limit model of the boundary problem of 2D incompressible Euler equation. For the Cauchy problem of the incompressible Euler equation, it is locally well-posed in $\mathcal{C}_T(B^1_{\infty,1}(\mathbb{R}^{2}))$ \cite{gly} and locally ill-posed in $L^{\infty}_T(\mathcal{C}^{0,1}(\mathbb{R}^{2}))$ (norm inflation) \cite{boli}. Then, the fact
\begin{align}\label{c01}
B^{1+\epsilon}_{\infty,1}\hookrightarrow B^{1}_{\infty,1}\cap B^{1}_{\infty,\infty,1}\hookrightarrow B^{1}_{\infty,1}\hookrightarrow \mathcal{C}^{0,1},\qquad\forall \epsilon>0
\end{align}
implies that $B^{1}_{\infty,1}(\mathbb{R}^{2})$ is the critical Besov space for the incompressible Euler equation where $B^{1}_{\infty,\infty,1}$ is a Banach space equiped with the norm $\|f\|_{B^{1}_{\infty,\infty,1}}=\sup\limits_{j}j\cdot2^{j}\|\Delta_{j}f\|_{L^{\infty}}$. For the CH equation, we see from Theorem \ref{wellch}-\ref{c01ill} that it is locally well-posed in $\mathcal{C}_T(B^{1+\epsilon}_{\infty,1}(\mathbb{R}))$ and locally ill-posed in $C_T(\mathcal{C}^{0,1}(\mathbb{R}))$. Analogously, a nature problem is
\quad\\

\textbf{H}:{\it ~~Whether the problem \eqref{ch} is locally well-posed or not in $B^1_{\infty,1}(\mathbb{R})\cap B^{1}_{\infty,\infty,1}(\mathbb{R})$ or $B^1_{\infty,1}(\mathbb{R})$?}
\quad\\

In this paper, we aim to solve this problem. The main difficulty is the force term $-\partial_x(1-\partial_{xx})^{-1}(u^2+\frac{u_x^2}{2})$ (we will focus on $-\partial_x(1-\partial_{xx})^{-1}(\frac{u_x^2}{2})$ in below since $-\partial_x(1-\partial_{xx})^{-1}(u^2)$ is a lower order term). In $B^{1}_{\infty,1}(\mathbb{R})\cap B^{1}_{\infty,\infty,1}(\mathbb{R})$, we can easliy get the following estimate
$$\|-\partial_x(1-\partial_{xx})^{-1}(u^2+\frac{u_x^2}{2})\|_{B^{1}_{\infty,1}\cap B^{1}_{\infty,\infty,1}}\leq C \|u\|^2_{B^{1}_{\infty,1}\cap B^{1}_{\infty,\infty,1}}.$$
So the local well-posedness for the Cauchy problem of the CH equation in $B^{1}_{\infty,1}(\mathbb{R})\cap B^{1}_{\infty,\infty,1}(\mathbb{R})$ is obvious, see Theorem \ref{well}. But only for $B^{1}_{\infty,1}(\mathbb{R})$, since we don't know whether $B^{0}_{\infty,1}(\mathbb{R})$ is a Banach algebra or not, we can't obtain a priori estimate of $\|-\partial_x(1-\partial_{xx})^{-1}(u^2+\frac{u_x^2}{2})\|_{B^{1}_{\infty,1}}$ (This problem doesn't exist for the incompressible Euler equation, since the divergence free condition makes sense). The good news is, we construct a counter example to present that $B^{0}_{\infty,1}(\mathbb{R})$ is not a Banach algebra, even $\|f^2\|_{B^{0}_{\infty,1}}$ is not bounded when $f\in B^{0}_{\infty,1}(\mathbb{R})$. Note that
for any $u_0\in B^1_{\infty,1}(\mathbb{R})$, we have
$$E_0:=-\partial_x(1-\partial_{xx})^{-1}(\frac{u_{0x}^2}{2})\in B^1_{\infty,1}(\mathbb{R}) \Longleftrightarrow u_{0x}^2\in B^0_{\infty,1}(\mathbb{R}),\quad
-\partial_{xx}(1-\partial_{xx})^{-1}={\rm Id}-(1-\partial_{xx})^{-1}.$$
So we conclude that the CH equation is ill-posed in $\mathcal{C}_T(B^1_{\infty,1}(\mathbb{R}))$, see Theorem \ref{ill}.\\

Now, we present our two main theorems which answer the problem \textbf{H}:
\begin{theo}\label{well}
Let $u_0\in B^{1}_{\infty,1}(\mathbb{R})\cap B^{1}_{\infty,\infty,1}(\mathbb{R})\big(B^{1+\epsilon}_{\infty,1}(\mathbb{R})\hookrightarrow B^{1}_{\infty,1}(\mathbb{R})\cap B^{1}_{\infty,\infty,1}(\mathbb{R}),\ \forall \epsilon>0\big)$. Then there exists a time $T>0$ such that the Cauchy problem \eqref{ch} is locally well-posed in $\mathcal{C}_T(B^{1}_{\infty,1}(\mathbb{R})\cap B^{1}_{\infty,\infty,1}(\mathbb{R}))\cap \mathcal{C}^1_T(B^{0}_{\infty,1}(\mathbb{R})\cap B^{0}_{\infty,\infty,1}(\mathbb{R}))$ in the sense of Hadamard.
\end{theo}

\begin{theo}\label{ill}
For any $10<N\in\mathbb{N}^{+}$ large enough, there exists a  $u_{0}\in\mathcal{C}^{\infty}(\mathbb{R})$ such that the following hold:
\begin{itemize}
\item [{\rm(1)}] $\|u_{0}\|_{B^{1}_{\infty,1}}\leq CN^{-\frac{1}{10}};$
\item [{\rm(2)}] There is a unique solution $u\in \mathcal{C}_{T}\big(\mathcal{C}^{\infty}(\mathbb{R})\big)$ to the Cauchy problem \eqref{ch} with a time $T\leq\frac{2}{N^\frac{1}{2}};$
\item [{\rm(3)}]
There exists a time $t_{0}\in[0,T]$ such that $\|u(t_{0})\|_{B^{1}_{\infty,1}}\geq \ln N$.
\end{itemize}
\end{theo}
\begin{rema}
From Theorem \ref{c01ill} and Theorem \ref{ill}, we know that the Cauchy problem \eqref{ch} is ill-posed in $\mathcal{C}^{0,1}(\mathbb{R})$ and $B^1_{\infty,1}(\mathbb{R})$ respectively. The difference is,  in $B^1_{\infty,1}(\mathbb{R})$, we present the norm inflation (the norm inflation implies the discontinuous dependence) and hence the ill-posedness for the Cauchy problem \eqref{ch}, which means that the Cauchy problem \eqref{ch} is ill-posed in $B^1_{\infty,1}(\mathbb{R})$ due to the norm inflation; in $\mathcal{C}^{0,1}(\mathbb{R})$, there is no norm inflation but the continuous dependence is still broken, which means that the Cauchy problem \eqref{ch} is ill-posed in $\mathcal{C}^{0,1}(\mathbb{R})$ because of the discontinuous dependence.
\end{rema}
\begin{rema}
Since the incompressible Euler equation is locally well-posed in $\mathcal{C}_T(B^1_{\infty,1}(\mathbb{R}^{2}))$ \cite{gly} and locally ill-posed in $L^{\infty}_T(\mathcal{C}^{0,1}(\mathbb{R}^{2}))$ \cite{boli} (norm inflation), but for the CH equation we obtain the local ill-posedness in $\mathcal{C}_T(B^1_{\infty,1}(\mathbb{R}))$ (norm inflation) and $L^{\infty}_T(\mathcal{C}^{0,1}(\mathbb{R}))$ (only the continuous dependence is broken). This interesting fact illustrates that there is a nature difference between the two equations.
\end{rema}

In sum, from Theorem \ref{wellch}--Theorem \ref{ill} we know that $B^{1}_{\infty,1}(\mathbb{R})$ is the critical Besov space for the CH equation, and the local well-posedness and ill-posedness for the CH equation in all critical Besov spaces have been completed, see the following
\begin{align*}
\underset{}{\overset{\text{well-posed}}{B^{1+\frac{1}{p}}_{p,1}(\mathbb{R})}~\text{with}~p<\infty}~\hookrightarrow\underset{\text{norm inflation}}{\overset{\text{ill-posed}}{B^{1}_{\infty,1}(\mathbb{R})}}\hookrightarrow\underset{\text{no norm inflation}}{\overset{\text{ill-posed}}{\mathcal{C}^{0,1}(\mathbb{R}).}}
\end{align*}
$$\text{and}$$
\begin{align*} 
\underset{}{\overset{\text{well-posed}}{B^{1+\epsilon}_{\infty,1}(\mathbb{R})} }~\text{with any}~\epsilon>0~\hookrightarrow\underset{\text{norm inflation}}{\overset{\text{ill-posed}}{B^{1}_{\infty,1}(\mathbb{R})}}\hookrightarrow\underset{\text{no norm inflation}}{\overset{\text{ill-posed}}{\mathcal{C}^{0,1}(\mathbb{R}).}}
\end{align*}

Finally, by the proof of Theorem \ref{ill}, we prove a norm inflation by constructing an special initial data $u_{0}$ satisfies $u_0\in B^1_{\infty,1}(\mathbb{R})$ but $u^2_{0x}\notin B^{0}_{\infty,1}(\mathbb{R})$, we now show that this condition is necessary.
\begin{theo}\label{non}
Let $u_{0}\in B^{1}_{\infty,1}(\mathbb{R})$. If $u^2_{0x}\in B^{0}_{\infty,1}(\mathbb{R})$, then there exists a lifespan $T>0$ such that the Cauchy problem \eqref{ch} has a unique solution $u(t,x)\in \mathcal{C}_T(B^{1}_{\infty,1}(\mathbb{R}))\cap  \mathcal{C}^{1}_T(B^{0}_{\infty,1}(\mathbb{R}))$. This means the norm inflation will not occur.
\end{theo}

The rest of our paper is as follows. In Section 2, we introduce some preliminaries which will be used in the sequel. In Section 3, we establish the well-posedness for the Cauchy problem \eqref{ch} in $B^{1}_{\infty,1}(\mathbb{R})\cap B^{1}_{\infty,\infty,1}(\mathbb{R})$.
In Section 4, we first present the norm inflation and hence the ill-posedness for the Cauchy problem \eqref{ch} in $B^{1}_{\infty,1}(\mathbb{R})$ by  choosing an special initial data $u_0\in B^{1}_{\infty,1}(\mathbb{R})$ but $u^2_{0x}\notin B^{0}_{\infty,1}(\mathbb{R})$ (an example implies $B^{0}_{\infty,1}(\mathbb{R})$ is not a Banach algebra). Then, we prove that this condition is necessary. That is, if $u^2_{0x}\in B^{0}_{\infty,1}(\mathbb{R})$ holds, the Camassa-Holm equation has a unique solution $u(t,x)\in \mathcal{C}_T(B^{1}_{\infty,1}(\mathbb{R}))\cap \mathcal{C}^{1}_T(B^{0}_{\infty,1}(\mathbb{R}))$ and the norm inflation will not occur.

\section{Preliminaries}
\par
In this section, we recall some basic properties about the Littlewood-Paley theory and Besov spaces, which can be found in \cite{book}.

\begin{prop}[Bernstein's inequalities, see \cite{book}]\label{berns}
Let $\mathfrak{B}$ be a ball and $\mathfrak{C}$ be an annulus. A constant $C>0$ exists such that for all $k\in\mathbb{N},\ 1\leq  p\leq q\leq\infty$, and any function $f\in L^{p}(\mathbb{R})$, we have
\begin{align*}
&{\rm Supp}{(\mathcal{F}{f})}\subset\lambda\mathfrak{B}\Rightarrow\|D^{k}f\|_{L^{q}}=\sup\limits_{|\alpha|=k}\|\partial^{\alpha}f\|_{L^{q}}\leq C^{k+1}\lambda^{k+d(\frac1p-\frac1q)}\|f\|_{L^{p}},\\
&{\rm Supp}{(\mathcal{F}{f})}\subset\lambda\mathfrak{C}\Rightarrow C^{-k-1}\lambda^{k}\|f\|_{L^{p}}\leq\|D^{k}f\|_{L^{q}}\leq C^{k+1}\lambda^{k}\|f\|_{L^{p}}.
\end{align*}	
\end{prop}
Let $\chi$ and $\varphi$ be a radical, smooth, and valued in the interval $[0,1]$, belonging respectively to $\mathcal{D}(\mathcal{B})$ and $\mathcal{D}(\mathcal{C})$, where $\mathcal{B}=\{\xi\in\mathbb{R}^d:|\xi|\leq\frac 4 3\},\ \mathcal{C}=\{\xi\in\mathbb{R}^d:\frac 3 4\leq|\xi|\leq\frac 8 3\}$.
Denote $\mathcal{F}$ by the Fourier transform and $\mathcal{F}^{-1}$ by its inverse.

The nonhomogeneous dyadic blocks $\Delta_j$ and low-frequency cut-off operators $S_j$ are defined as
\begin{equation*}
\left\{\begin{array}{ll}
\Delta_j u=0,\ \text{for $j\leq -2;$}\
\Delta_{-1} u=\mathcal{F}^{-1}(\chi\mathcal{F}u);\
\Delta_j u=\mathcal{F}^{-1}(\varphi(2^{-j}\cdot)\mathcal{F}u),\ \text{for $j\geq0,$}\\
S_j u=\sum\limits_{j'<j}\Delta_{j'}u.
\end{array}\right.
\end{equation*}

Let $s\in\mathbb{R},\ 1\leq p,\ r\leq\infty.$ The nonhomogeneous Besov spaces $B^s_{p,r}(\mathbb{R}^d)$ are defined by
\begin{align*}
B^s_{p,r}=B^s_{p,r}(\mathbb{R}^d)=\Big\{u\in S'(\mathbb{R}^d):\|u\|_{B^s_{p,r}}=\big\|(2^{js}\|\Delta_j u\|_{L^p})_j \big\|_{l^r(\mathbb{Z})}<\infty\Big\}.
\end{align*}

The nonhomogeneous Bony's decomposition is defined by
$uv=T_{u}v+T_{v}u+R(u,v)$ with
$$T_{u}v=\sum_{j}S_{j-1}u\Delta_{j}v,\ \ R(u,v)=\sum_{j}\sum_{|j'-j|\leq 1}\Delta_{j}u\Delta_{j'}v.$$



\begin{prop}[See \cite{book}]\label{Besov} Let $s,\ s_{1},\ s_{2}\in\mathbb{R}$ and $1\leq p,\ p_{1},\ p_{2},\ r,\ r_{1},\ r_{2}\leq+\infty$.
\begin{itemize}
\item [{\rm(1)}] If $p_1\leq p_2$ and $r_1\leq r_2$, then $ B^s_{p_1,r_1}\hookrightarrow B^{s-d(\frac 1 {p_1}-\frac 1 {p_2})}_{p_2,r_2}$.
\item [{\rm(2)}] If $s_1<s_2$ and $r_1\leq r_2$, then the embedding $B^{s_2}_{p,r_2}\hookrightarrow B^{s_1}_{p,r_1}$ is locally compact.
\item [{\rm(3)}] Let $m\in\mathbb{R}$ and $f$ be a $S^m$-mutiplier $($i.e. f is a smooth function and satisfies that $\forall\ \alpha\in\mathbb{N}^d$,
$\exists\ C=C(\alpha)$ such that $|\partial^{\alpha}f(\xi)|\leq C(1+|\xi|)^{m-|\alpha|},\ \forall\ \xi\in\mathbb{R}^d)$.
Then the operator $f(D)=\mathcal{F}^{-1}(f\mathcal{F})$ is continuous from $B^s_{p,r}$ to $B^{s-m}_{p,r}$.
\end{itemize}

\end{prop}

\begin{prop}[See \cite{book}]\label{bony}
Let $1\leq p,\ r,\ p_{1},\ p_{2},\ r_{1},\ r_{2}\leq+\infty$.
\begin{itemize}
\item [{\rm(1)}] For $s\in\mathbb{R},\ t<0$, a constant $C$ exists which satifies the following inequalities
\begin{align*}
&\|T_{f}g\|_{B^{s}_{p,r}}\leq C\|f\|_{L^{\infty}}\|g\|_{B^{s}_{p,r}};\\
&\|T_{f}g\|_{B^{s+t}_{p,r}}\leq C\|f\|_{B^{t}_{p,r_{1}}}\|g\|_{B^{s}_{p,r_{2}}}\quad\text{with $\frac1r\stackrel{\rm{def}}{=}\min\Big\{1,\frac{1}{r_{1}}+\frac{1}{r_{2}}\Big\}$}.	
\end{align*}	
\item [{\rm(2)}] For $s_{1},\ s_{2}\in\mathbb{R}$ satisfying $s_{1}+s_{2}>0$ and $\frac{1}{p}\stackrel{\rm{def}}{=}\frac{1}{p_{1}}+\frac{1}{p_{2}}\leq1,\ \frac{1}{r}\stackrel{\rm{def}}{=}\frac{1}{r_{1}}+\frac{1}{r_{2}}\leq1$, then
\begin{align*}
\|R(f,g)\|_{B^{s_{1}+s_{2}}_{p,r}}\leq C\|f\|_{B^{s_{1}}_{p_{1},r_{1}}}\|g\|_{B^{s_{2}}_{p_{2},r_{2}}}.	
\end{align*}	
\end{itemize}
\end{prop}

For the ill-posedness of the Cauchy problem \eqref{ch} in $B^{1}_{\infty,1}(\mathbb{R})$, we need some estimates in space $B^{0}_{\infty,\infty,1}(\mathbb{R})$ with the norm $\|f\|_{B^{0}_{\infty,\infty,1}}=\sup\limits_{j}\|\Delta_jf\|_{L^{\infty}}\cdot j$.
\begin{lemm}\label{b01}
For any $f\in B^{0}_{\infty,1}\cap B^{0}_{\infty,\infty,1}$, we have
\begin{align*}
&\|f^{2}\|_{B^{0}_{\infty,\infty,1}}\leq C\|f\|_{B^{0}_{\infty,1}}\|f\|_{B^{0}_{\infty,\infty,1}},\\
&\|f^{2}\|_{B^{0}_{\infty,1}}\leq C\|f\|_{B^{0}_{\infty,1}}\|f\|_{B^{0}_{\infty,\infty,1}}.
\end{align*}
\end{lemm}
\begin{lemm}\label{rj}
Define $R_{j}=\Delta_j(fg_{x})-f\Delta_jg_{x}$. Then we have
\begin{align*}
&\sup\limits_{j}\big(\|R_{j}\|_{L^{\infty}}\cdot j\big)\leq
C\|f_{x}\|_{B^{0}_{\infty,1}}\|g\|_{B^{0}_{\infty,\infty,1}}.\\
&\sum\limits_{j}\|R_{j}\|_{L^{\infty}}\leq C\|f_{x}\|_{B^{0}_{\infty,1}}\|g\|_{B^{0}_{\infty,1}\cap B^{0}_{\infty,\infty,1}}.\\
&\sum\limits_{j}2^{j}\|R_{j}\|_{L^{\infty}}\leq C\|f_{x}\|_{B^{0}_{\infty,1}}\|g\|_{B^{1}_{\infty,1}}.
\end{align*}
\end{lemm}
Lemma \ref{b01} and Lemma \ref{rj} can be proved by the Bony's decomposition and some similar calculations as Lemma 2.100 in \cite{book} and here we omit it.

In the paper, we also need some estimates for the following Cauchy problem of the 1-D transport equation:
\begin{equation}\label{transport}
\left\{\begin{array}{l}
f_t+v\partial_{x}f=g,\ x\in\mathbb{R},\ t>0, \\
f(0,x)=f_0(x),\ x\in\mathbb{R}.
\end{array}\right.
\end{equation}

\begin{lemm}[See \cite{book,liy}]\label{existence}
Let $1\leq p,\ r\leq\infty$ and $\theta> -\min(\frac 1 {p}, \frac 1 {p'}).$ Suppose $f_0\in B^{\theta}_{p,r},$ $g\in L^1(0,T;B^{\theta}_{p,r}),$ and $v\in L^\rho(0,T;B^{-M}_{\infty,\infty})$ for some $\rho>1$ and $M>0,$ and
\begin{align*}
\begin{array}{ll}
\partial_{x}v\in L^1(0,T;B^{\frac 1 {p}}_{p,\infty}\cap L^{\infty}), &\ \text{if}\ \theta<1+\frac 1 {p}, \\
\partial_{x}v\in L^1(0,T;B^{\theta}_{p,r}),\ &\text{if}\ \theta=1+\frac{1}{p},\ r>1, \\
\partial_{x}v\in L^1(0,T;B^{\theta-1}_{p,r}), &\ \text{if}\ \theta>1+\frac{1}{p}\ (or\ \theta=1+\frac 1 {p},\ r=1).
\end{array}	
\end{align*}
Then the problem \eqref{transport} has a unique solution $f$ in
\begin{itemize}
\item [-] the space $C([0,T];B^{\theta}_{p,r}),$ if $r<\infty,$
\item [-] the space $\big(\bigcap_{{\theta}'<\theta}C([0,T];B^{{\theta}'}_{p,\infty})\big)\bigcap C_w([0,T];B^{\theta}_{p,\infty}),$ if $r=\infty.$
\end{itemize}
\end{lemm}

\begin{lemm}[See \cite{book,liy}]\label{priori estimate}
Let $1\leq p,\ r\leq\infty$ and $\theta>-\min(\frac{1}{p},\frac{1}{p'}).$ There exists a constant $C$ such that for all solutions $f\in L^{\infty}(0,T;B^{\theta}_{p,r})$ of \eqref{transport} with initial data $f_0$ in $B^{\theta}_{p,r}$ and $g$ in $L^1(0,T;B^{\theta}_{p,r}),$ we have, for a.e. $t\in[0,T],$
$$ \|f(t)\|_{B^{\theta}_{p,r}}\leq \|f_0\|_{B^{\theta}_{p,r}}+\int_0^t\|g(t')\|_{B^{\theta}_{p,r}}{\ud}t'+\int_0^t V'(t')\|f(t')\|_{B^{\theta}_{p,r}}{\ud}t' $$
or
$$ \|f(t)\|_{B^{\theta}_{p,r}}\leq e^{CV(t)}\Big(\|f_0\|_{B^{\theta}_{p,r}}+\int_0^t e^{-CV(t')}\|g(t')\|_{B^{\theta}_{p,r}}{\ud}t'\Big) $$
with
\begin{equation*}
V'(t)=\left\{\begin{array}{ll}
\|\partial_{x}v(t)\|_{B^{\frac 1 p}_{p,\infty}\cap L^{\infty}},\ &\text{if}\ \theta<1+\frac{1}{p}, \\
\|\partial_{x}v(t)\|_{B^{\theta}_{p,r}},\ &\text{if}\ \theta=1+\frac{1}{p},\ r>1, \\
\|\partial_{x}v(t)\|_{B^{\theta-1}_{p,r}},\ &\text{if}\ \theta>1+\frac{1}{p}\ (\text{or}\ \theta=1+\frac{1}{p},\ r=1).
\end{array}\right.
\end{equation*}
If $\theta>0$, then there exists a constant $C=C(p,r,\theta)$ such that the following statement holds
\begin{align*}
\|f(t)\|_{B^{\theta}_{p,r}}\leq \|f_0\|_{B^{\theta}_{p,r}}+\int_0^t\|g(\tau)\|_{B^{\theta}_{p,r}}{\ud}\tau+C\int_0^t \Big(\|f(\tau)\|_{B^{\theta}_{p,r}}\|\partial_{x}v(\tau)\|_{L^{\infty}}+\|\partial_{x}v(\tau)\|_{B^{\theta-1}_{p,r}}\|\partial_{x}f(\tau)\|_{L^{\infty}}\Big){\ud}\tau. 	
\end{align*}
In particular, if $f=av+b,\ a,\ b\in\mathbb{R},$ then for all $\theta>0,$ $V'(t)=\|\partial_{x}v(t)\|_{L^{\infty}}.$
\end{lemm}

\section{Well-posedness}
\par
In this section, we main study the local well-posedness for the CH equation in subcritical spaces (see Theorem \ref{well}):
$$B^{1+\epsilon}_{\infty,1}\hookrightarrow B^{1}_{\infty,1}\cap B^{1}_{\infty,\infty,1}\hookrightarrow B^{1}_{\infty,1}\hookrightarrow \mathcal{C}^{0,1}.$$

To prove Theorem \ref{well}, we must recall a useful lemma.
First we consider the following Cauchy problem for a general abstract equation
\begin{equation}\label{abstract}
\left\{\begin{array}{ll}
\partial_{t}u+A(u)\partial_{x}u=F(u),&\quad t>0,\quad x\in\mathbb{R},\\
u(t,x)|_{t=0}=u_0(x),&\quad x\in\mathbb{R}
\end{array}\right.
\end{equation}
where $A(u)$ is a polynomial of $u$, $F$ is called a `good operator' such that for any $\varphi\in \mathcal{C}^{\infty}_0$ and any $\epsilon>0$ small enough,
$$\text{if}\quad u_n\varphi\rightarrow u\varphi\quad \text{in}\quad B^{1+\frac{1}{p}-\epsilon}_{p,1},\quad \text{then}\quad  \langle F(u_n),\varphi\rangle\longrightarrow\langle F(u),\varphi\rangle.$$
This definition is reasonable for the Camassa-Holm type equations. For example, it's easy to prove that $F(u)=-\partial_{x}(1-\partial_{xx})^{-1}\Big(u^2+\frac{u^2_x}{2}\Big)$ is a `good operator' by an approximation argument, since $\mathcal{C}^{\infty}_0$ is dense in $\mathcal{S}$.

The associated Lagrangian scale of \eqref{abstract} is the following initial valve problem
\begin{equation}\label{ODE}
\left\{\begin{array}{ll}
\frac{{\ud}}{{\ud}t}y(t,\xi)=A(u)\big(t,y(t,\xi)\big),&\quad t>0,\quad \xi\in\mathbb{R},\\
y(0,\xi)=\xi,&\quad \xi\in\mathbb{R}.
\end{array}\right.
\end{equation}
Introduce the new variable $U(t,\xi)=u\big(t,y(t,\xi)\big)$. Then, \eqref{abstract} becomes
\begin{equation}\label{lagrange}
\left\{\begin{array}{ll}
U_t=\Big(F(u)\Big)(t,y(t,\xi)):=\widetilde{F}(U,y),&\quad t>0,\quad \xi\in\mathbb{R},\\
U(t,\xi)|_{t=0}=U_0(\xi)=u_0(\xi),&\quad \xi\in\mathbb{R}.
\end{array}\right.
\end{equation}
Here is the useful lemma:
\begin{lemm}[See \cite{yyg}]\label{lagrabstr}
Let $u_0\in B^{1+\frac{1}{p}}_{p,1}$ with $1\leq p<\infty$ and $k\in\mathbb{N}^{+}$. Suppose $F$ is a `good operator' and $F,\ \widetilde{F}$ satisfy the following conditions:
\begin{align}
&\|F(u)\|_{B^{1+\frac{1}{p}}_{p,1}}\leq C\Big(\|u\|^{k+1}_{B^{1+\frac{1}{p}}_{p,1}}+1\Big);\label{fkl}\\
&\|\widetilde{F}(U,y)-\widetilde{F}(\bar{U},\bar{y})\|_{W^{1,\infty}\cap W^{1,p}}\leq C\Big(\|U-\bar{U}\|_{W^{1,\infty}\cap W^{1,p}}+\|y-\bar{y}\|_{W^{1,\infty}\cap W^{1,p}}\Big);\label{wideFkl}\\
&\|F(u)-F(\bar{u})\|_{B^{1+\frac{1}{p}}_{p,1}}\leq C\|u-\bar{u}\|_{B^{1+\frac{1}{p}}_{p,1}}\Big(\|u\|^{k}_{B^{1+\frac{1}{p}}_{p,1}}+\|\bar{u}\|^{k}_{B^{1+\frac{1}{p}}_{p,1}}+1\Big).\label{Fkl}
\end{align}
Then, there exists a time $T>0$ such that
\begin{itemize}
\item [\rm{(1)}] Existence: If \eqref{fkl} holds, then \eqref{abstract} has a solution $u\in E^p_T:=\mathcal{C}_{T}\big(B^{1+\frac{1}{p}}_{p,1}\big)\cap \mathcal{C}^{1}_{T}\big(B^{\frac{1}{p}}_{p,1}\big);$
\item [\rm{(2)}] Uniqueness: If \eqref{fkl} and \eqref{wideFkl} hold, then the solution of \eqref{abstract} is unique$;$
\item [\rm{(3)}] Continuous dependence: If \eqref{fkl}--\eqref{Fkl} hold, then the solution map is continuous from any bounded subset of $ B^{1+\frac{1}{p}}_{p,1}$ to $ \mathcal{C}_{T}\big(B^{1+\frac{1}{p}}_{p,1}\big)$.
\end{itemize}
That is, the problem \eqref{abstract} is locally well-posed in the sense of Hadamard.
\end{lemm}

\begin{proof}[\rm\textbf{The proof of Theorem \ref{well}:}]
For $u_0\in B^{1}_{\infty,1}\cap B^{1}_{\infty,\infty,1}$, since
$$\|u^2_x\|_{B^{0}_{\infty,1}\cap B^{0}_{\infty,\infty,1}}\leq C\|u_x\|_{B^{0}_{\infty,1}\cap B^{0}_{\infty,\infty,1}}^2,$$
taking the similar proof of Lemma \ref{lagrabstr}, one can easily obtain \eqref{fkl}, \eqref{wideFkl} and \eqref{Fkl}. Therefore, the local well-posedness for the Cauchy problem \eqref{ch} of the CH equation is obvious, and here we omit the details.
\end{proof}

\section{Ill-posedness}
\par
\subsection{Norm inflation}
In this subsection, we main investigate the ill-posedness for the Cauchy problem \eqref{ch} of the CH equation in $B^{1}_{\infty,1}$. To begin with, we construct a vital counter example which implies the space $B^{0}_{\infty,1}$ is not a Banach algebra.

\begin{lemm}
The space $B^{0}_{\infty,1}$ is not a Banach algebra. That is, there exist $f,\ g\in B^{0}_{\infty,1}$ such that
$\|fg\|_{B^{0}_{\infty,1}} \nleq C\|f \|_{B^{0}_{\infty,1}}\|g\|_{B^{0}_{\infty,1}}$.
\end{lemm}
\begin{proof}
The lemma may be a well-known result in some classical books about functional analysis. However, in order to use this lemma to the CH equation, we must prove a stronger case: $f=g$. That is, there exists $f \in B^{0}_{\infty,1}$ such that
$\|f^2\|_{B^{0}_{\infty,1}} \nleq \|f \|^2_{B^{0}_{\infty,1}}$, which also implies the lemma.

Now, choose
\begin{align}
u_{0}(x)=-(1-\partial_{xx})^{-1}\partial_{x}\Big[\cos2^{N+5}x\cdot \big(1+N^{-\frac{1}{10}}S_{N}h(x)\big)\Big]N^{-\frac{1}{10}}\label{u0}
\end{align}
where $S_{N}f=\sum\limits_{-1\leq j<N}\Delta_{j}f$ and $h(x)=1_{x\geq0}(x)$.\\
Since
\begin{align}
\mathcal{F}\big(\cos2^{N+5}x\big)\approx\delta(\xi+2^{N+5})+\delta(\xi-2^{N+5}),\label{cos}
\end{align}
we know that
\begin{align*}
\|\cos2^{N+5}x\|_{B^{0}_{\infty,1}}\leq C.
\end{align*}

On the one hand, note that the Fourier transform of $S_{N}h$ is supported in ball $2^{N}B$ where $B$ is a ball, the Fourier transform of $\cos2^{N+5}x\cdot S_{N}h$ is supported in annulus $\pm[2^{N+5}-2^{N},2^{N+5}+2^{N}]$. It follows that
\begin{align*}
\Delta_{j}\big(\cos2^{N+5}x\cdot S_{N}h\big)=\cos2^{j}x\cdot S_{j-5}h,\quad j\approx N+5.
\end{align*}
Therefore,
\begin{align*}
\|\cos2^{N+5}x\cdot S_{N}h\|_{B^{0}_{\infty,1}}\leq C
\end{align*}
which infers that
\begin{align*}
\|u_{0}\|_{B^{1}_{\infty,1}}\big(\approx\|u_{0}\|_{L^{\infty}}+\|u_{0x}\|_{B^{0}_{\infty,1}}\big)\leq CN^{-\frac{1}{10}}.
\end{align*}
It then turns out that
\begin{align}
\|u_{0}\|_{L^{\infty}},\ \|u_{0x}\|_{B^{0}_{\infty,1}}\leq CN^{-\frac{1}{10}}.\label{u0xb0}
\end{align}
Thanks to \eqref{u0xb0} and the definition of the space $B^{0}_{\infty,\infty,1}$, we also see
\begin{align}
\|u_{0}\|_{B^{1}_{\infty,\infty,1}}=\sup\limits_{j}j2^{j}\|\Delta_{j}u_{0}\|_{L^{\infty}}\approx\sup\limits_{j}j\|\Delta_{j}u_{0x}\|_{L^{\infty}}\leq C N^{\frac{9}{10}},\qquad j\approx N+5.\label{b1ww}
\end{align}

On the other hand, since
\begin{align*}
u_{0x}\approx\Big[\cos2^{N+5}x\cdot (1+N^{-\frac{1}{10}}S_{N}h)+R_{L}^{1}\Big]N^{-\frac{1}{10}}
\end{align*}
where $\mathcal{R}_{L}^{1}=-(1-\partial_{xx})^{-1}\big[\cos2^{N+5}x\cdot (1+N^{-\frac{1}{10}}S_{N}h)\big]$ is a lower order term ($\mathcal{R}_{L}^{i}$ is noted here and below, without causing any confusion, the lower order terms). Note that $(1-\partial_{xx})^{-1}$ is a $S^{-2}$ operator, we discover
\begin{align}
\|\mathcal{R}_{L}^{1}\|_{L^{\infty}}\leq C\|\mathcal{R}_{L}^{1}\|_{B^{0}_{\infty,1}}\leq C\|\mathcal{R}_{L}^{1}\|_{B^{1}_{\infty,1}}\leq C\Big\|\cos2^{N+5}x\cdot \big(1+N^{-\frac{1}{10}}S_{N}h\big)\Big\|_{B^{0}_{\infty,1}}\leq C.\label{rj1}
\end{align}
Then using the fact $\cos^{2}x=\frac{1+\cos2x}{2}$ we find
\begin{align}
u_{0x}^{2}\approx&\Big[\cos^{2}2^{N+5}x\cdot (1+2N^{-\frac{1}{10}}S_{N}h+N^{-\frac{1}{5}}(S_{N}h)^{2})+\mathcal{R}_{L}^{2}\Big]N^{-\frac{2}{10}}\notag\\
\approx&\Big[N^{-\frac{1}{5}}(S_{N}h)^{2}+N^{-\frac{1}{10}}S_{N}h+\mathcal{R}_{L}^{3}\Big]N^{-\frac{1}{5}}\label{u0x2}
\end{align}
where
\begin{align*}
&\mathcal{R}_{L}^{2}=(\mathcal{R}_{L}^{1})^{2}+\cos2^{N+5}x\cdot \big(1+N^{-\frac{1}{10}}S_{N}h\big)\cdot\mathcal{R}_{L}^{1},\\
&\mathcal{R}_{L}^{3}=\cos^{2}2^{N+5}x+\cos2^{N+6}x\cdot \big(N^{-\frac{1}{10}}S_{N}h++N^{-\frac{1}{5}}(S_{N}h)^{2}\big)+\mathcal{R}_{L}^{2}.	
\end{align*}
As the Fourier transform of $S_{N}h$ is supported in $2^{N}B$, then the Fourier transform of $(S_{N}h)^{2}$ is supported in $2^{2N}B$ and the Fourier transform of $\cos2^{N+6}x\cdot(S_{N}h)^{2}$ is supported in annulus $\pm[2^{2N}-2^{N+6},2^{2N}+2^{N+6}]$, we thus have
\begin{align*}
&\Delta_{j}\big(\cos2^{N+6}x\cdot(S_{N}h)^{2}\big)=\cos2^{j}x\cdot(S_{j}h)^{2},\qquad j\approx 2N,\\
&\|\cos2^{N+6}x\cdot(S_{N}h)^{2}\|_{B^{0}_{\infty,1}}\leq C.
\end{align*}
It follows from
\eqref{rj1} that
\begin{align*}
&\|(\mathcal{R}_{L}^{1})^{2}\|_{B^{0}_{\infty,1}}\leq C\|(\mathcal{R}_{L}^{1})^{2}\|_{B^{1}_{\infty,1}}
\leq C\|\mathcal{R}_{L}^{1}\|^{2}_{B^{1}_{\infty,1}}\leq C,\\
&\|\cos2^{N+5}x\cdot \big(1+N^{-\frac{1}{10}}S_{N}h(x)\big)\cdot\mathcal{R}_{L}^{1}\|_{B^{0}_{\infty,1}}\leq C.
\end{align*}
Hence,
\begin{align}
&\|\mathcal{R}_{L}^{2}\|_{B^{0}_{\infty,1}},\quad\|\mathcal{R}_{L}^{3}\|_{B^{0}_{\infty,1}}\leq C,\label{rj3}\\
&\|(S_{N}h)^{2}\|_{B^{0}_{\infty,1}}\leq\sum\limits_{j<2N+1}\|\Delta_{j}(S_{N}h)^{2}\|_{L^{\infty}}\leq C\sum\limits_{j<2N+1}\|(S_{N}h)^{2}\|_{L^{\infty}}\leq CN.\label{snh2}
\end{align}
Observe that
\begin{align*}
\Delta_{j}h=&\int_{\mathbb{R}}2^{j}\check{\varphi}(2^{j}(x-y))1_{y\geq0}(y){\ud}y\\
=&\int_{0}^{+\infty}2^{j}\check{\varphi}(2^{j}(x-y)){\ud}y\\	=&\int_{0}^{+\infty}\check{\varphi}(2^{j}x-\eta){\ud}\eta:=f(2^{j}x),
\end{align*}
we get
\begin{align*}
\|\Delta_{j}h\|_{L^{\infty}}=\|f(2^{j}\cdot)\|_{L^{\infty}}=\|f(\cdot)\|_{L^{\infty}}=\|\Delta_{0}h\|_{L^{\infty}}.
\end{align*}
It therefore turns out that
\begin{align}
\|S_{N}h\|_{B^{0}_{\infty,1}}=&\sum\limits_{j\leq N}\|\Delta_{j}\sum\limits_{k< N,|k-j|\leq1}\Delta_{k}h\|_{L^{\infty}}\notag\\
=&\sum\limits_{j\leq N}\|\Delta_{j}h\|_{L^{\infty}}=\sum\limits_{j\leq N}\|\Delta_{0}h\|_{L^{\infty}}\approx N\|\Delta_{0}h\|_{L^{\infty}}\approx C_{1}N.\label{snh}
\end{align}
\eqref{u0x2}, \eqref{rj3}, \eqref{snh2} and \eqref{snh} together yield
\begin{align}
\|u_{0x}^{2}\|_{B^{0}_{\infty,1}}\geq&\Big(\|N^{-\frac{1}{10}}S_{N}h\|_{B^{0}_{\infty,1}}-\|N^{-\frac{1}{5}}(S_{N}h)^{2}\|_{B^{0}_{\infty,1}}-C\Big)N^{-\frac{1}{5}}\notag\\
\geq& C_{1}NN^{-\frac{1}{10}}N^{-\frac{1}{5}}-CNN^{-\frac{1}{5}}N^{-\frac{1}{5}}-CN^{-\frac{1}{5}}\notag\\
\approx&\frac{C_{1}}{2}N^{\frac{7}{10}}\notag\\
\geq&CN^{\frac{3}{5}}.\label{fanli}
\end{align}
That is, we finally deduce
\begin{align}
\|u_{0x}\|_{B^{0}_{\infty,1}}\leq CN^{-\frac{1}{10}},\quad \|u_{0x}^{2}\|_{B^{0}_{\infty,1}}\geq CN^{\frac{3}{5}}\label{u0x2in}
\end{align}
which implies that the space $B^{0}_{\infty,1}$ is not a Banach algebra by setting $f=u_{0x}$.
\end{proof}

\begin{proof}[\rm{\textbf{The proof of Theorem \ref{ill}:}}] Let $u$ be a solution to the CH equation with the initial data $u_{0}$ defined as \eqref{u0}.
Set
\begin{align}
\frac{\ud}{\ud t}y(t,\xi)=u(t,y(t,\xi)),\quad y_{0}(\xi)=\xi. \label{char}
\end{align}
Owing to Theorem \ref{c01ill}, the CH equation has a solution $u(t,x)$ with the initial data $u_{0}$ in $\mathcal{C}^{0,1}$ such that
\begin{align}
\|u(t)\|_{\mathcal{C}^{0,1}}\leq C\|u_0\|_{\mathcal{C}^{0,1}}\leq CN^{-\frac{1}{10}},\quad \forall~ t\in[0,T_{0}]\label{uxinfty}
\end{align}
where $T_{0}<\frac{1}{4C\|u_{0}\|_{\mathcal{C}^{0,1}}},\ \|u_{0}\|_{\mathcal{C}^{0,1}}\leq C \|u_{0}\|_{B^1_{\infty,1}}\leq CN^{-\frac{1}{10}}$ and $C$ is a constant independent of $N$.

Therefore, according to \eqref{char} and \eqref{uxinfty}, we can find a $T_1>0$ sufficiently small such that $\frac{1}{2}\leq y_{\xi}(t)\leq2$ for any $t\in[0,\min\{T_0,T_1\}]$. Let $\bar{T}=\frac{2}{N^{\frac{1}{2}}}\leq \min\{T_0,T_1\}$ for $N>10$ large enough. To prove the norm inflation, we see that it suffices to prove there exists a time $t_{0}\in[0,\frac{2}{N^{\frac{1}{2}}}]$ such that $\|u_{x}(t_{0})\|_{B^{0}_{\infty,1}}\geq \ln N$ for $N>10$ large enough. Let us assume the opposite. Namely, we suppose that
\begin{align}
\sup\limits_{t\in[0,\frac{2}{N^{\frac{1}{2}}}]}\|u_{x}(t)\|_{B^{0}_{\infty,1}}<\ln N.\label{uxln}
\end{align}
Applying $\Delta_{j}$ and the Lagrange coordinates to Eq. \eqref{nonch}, and then integrating with respect to $t$, we get
\begin{align}
(\Delta_{j}u)\circ y=\Delta_{j}u_{0}+\int_{0}^{t}\underbrace{-R_{j}\circ y}_{I_{1}}+\underbrace{(\Delta_{j}\mathcal{R}_{L}^{4})\circ y}_{I_{2}}+\underbrace{\Delta_{j}E\circ y-\Delta_{j}E_{0}}_{I_{3}}{\ud}s+t\Delta_{j}E_{0}\label{delta}	
\end{align}
where
\begin{align*}
&R_{j}=\Delta_{j}(uu_{x})-u\Delta_{j}u_{x},\\
&\mathcal{R}_{L}^{4}=-\big[(1-\partial_{xx})^{-1}\partial_{x}(u^{2})\big],\\
&E(t,x)=-(1-\partial_{xx})^{-1}\partial_{x}(\frac{u^2_x}{2}).
\end{align*}
Let $T=\frac{1}{N^{\frac{1}{2}}}$ (Indeed $\forall T\in [\frac{1}{2N^{\frac{1}{2}}},\frac{3}{2N^{\frac{1}{2}}}]$ also makes sense). We will obtain the norm inflation by the following estimates:\\
${\rm(i)}$ Following the similar proof of Lemma 2.100 in \cite{book}, we see
\begin{align*}
\sum\limits_{j}2^{j}\|I_{1}\|_{L^{\infty}}\leq\sum\limits_{j}2^{j}\|R_{j}\|_{L^{\infty}}\leq \|u_{x}\|_{L^{\infty}}\|u\|_{B^{1}_{\infty,1}}\leq\|u\|_{\mathcal{C}^{0,1}}\cdot\ln N\leq \frac{C\ln N}{N^{\frac{1}{10}}}.	
\end{align*}
${\rm(ii)}$ According to the Bony's decomposition, we find
\begin{align*}
\sum\limits_{j}2^{j}\|I_{2}\|_{L^{\infty}}\leq\sum\limits_{j}2^{j}\|\Delta_{j}\mathcal{R}_{L}^{4}\|_{L^{\infty}}\leq C\|u\|_{L^{\infty}}\|u\|_{B^{1}_{\infty,1}}\leq\|u\|_{\mathcal{C}^{0,1}}\cdot\ln N\leq \frac{C\ln N}{N^{\frac{1}{10}}}.
\end{align*}
${\rm(iii)}$ Now we estimate $I_{3}$. Noting that $u(t,x)\in L^{\infty}_{T}(\mathcal{C}^{0,1})$ is a solution to the CH equation, then we have
\begin{equation}\label{E}
\left\{\begin{array}{l}
\frac{d}{dt}E+u\partial_{x}E=G(t,x),\quad t\in (0,T],\\
E(0,x)=E_{0}(x)=-(1-\partial_{xx})^{-1}\partial_{x}(\frac{u^2_{0x}}{2})
\end{array}\right.
\end{equation}
where $G(t,x)=\frac{u^3}{3}-u(1-\partial_{xx})^{-1}\big(\frac{u_x^2}{2}\big)-(1-\partial_{xx})^{-1}
\Big(\frac{u^3}{3}-\frac{1}{2}uu^2_x
-\partial_{x}\big[u_x(1-\partial_{xx})^{-1}(u^2+\frac{u_x^2}{2})\big]\Big)$. Since $(1-\partial_{xx})^{-1}$ is a $S^{-2}$ operator in nonhomogeneous Besov spaces, one can easily get
\begin{align}
\|G(t)\|_{B^{1}_{\infty,1}}\leq C\|u(t)\|^2_{\mathcal{C}^{0,1}}\|u(t)\|_{B^{1}_{\infty,1}}\leq CN^{-\frac{1}{5}}\ln N,\quad\forall~t\in(0,T].\label{nonlin}
\end{align}
Applying $\Delta_{j}$ and the Lagrange coordinates to \eqref{E} yields
\begin{align}\label{22}
(\Delta_{j}E)\circ y-\Delta_{j}E_0=\int_0^t\tilde{R}_{j}\circ y+ (\Delta_{j}G)\circ y{\ud}s	
\end{align}
where $\tilde{R}_{j}=u\partial_x\Delta_{j} E-\Delta_{j}\big(u\partial_xE\big)$. By Lemmas \ref{b01} and \ref{rj}, we discover
\begin{align}
\sum2^j
\|\tilde{R}_{j}\circ y\|_{L^{\infty}}=&\sum2^j
\|\tilde{R}_{j}\|_{L^{\infty}}\leq C\|u\|_{B^{1}_{\infty,1}}\|E\|_{B^{1}_{\infty,1}}\notag\\
\leq& C\|u\|_{B^{1}_{\infty,1}}\|u_{x}\|_{B^{0}_{\infty,1}}\|u_{x}\|_{B^{0}_{\infty,\infty,1}}\leq C(\ln N)^{2}\|u_{x}\|_{B^{0}_{\infty,\infty,1}}.\label{comm}
\end{align}
Thereby, we deduce that
\begin{align}
\sum\limits_{j}2^{j}\big\|\Delta_{j}E\circ y-\Delta_{j}E_0\|_{L^{\infty}}\leq&C\int_{0}^{t}\sum\limits_{j}2^{j}\|\tilde{R}_{j}\circ y\|_{L^{\infty}}+\sum\limits_{j}2^{j}\|\Delta_{j}G\circ y\|_{L^{\infty}}{\ud}s\notag\\
\leq&CT\cdot(\ln N)^{2}\cdot\|u_{x}\|_{L^{\infty}_{T}(B^{0}_{\infty,\infty,1})}+CT\cdot N^{-\frac{1}{5}}\cdot\ln N.\label{bern}
\end{align}
Moreover, note that $u_{x}$ solves
\begin{align*}
u_{xt}+uu_{xx}=u^{2}-\frac{1}{2}u_{x}^{2}-(1-\partial_{xx})^{-1}\big(u^{2}+\frac{1}{2}u_{x}^{2}\big).
\end{align*}
Then, following the similar proof of Lemma \ref{existence} and Lemma \ref{priori estimate}, we deduce
\begin{align}
\|u_{x}\|_{L^{\infty}_{T}(B^{0}_{\infty,\infty,1})}\leq&\|u_{x}\|_{L^{\infty}_{T}(B^{0}_{\infty,1}\cap B^{0}_{\infty,\infty,1})}\notag\\
\leq&\|u_{0x}\|_{B^{0}_{\infty,1}\cap B^{0}_{\infty,\infty,1}}+C\int_{0}^{T}\|u_{x}^{2}\|_{B^{0}_{\infty,1}\cap B^{0}_{\infty,\infty,1}}+\|u^{2}\|_{B^{0}_{\infty,1}\cap B^{0}_{\infty,\infty,1}}\ud\tau\notag\\
\leq&\|u_{0x}\|_{B^{0}_{\infty,1}\cap B^{0}_{\infty,\infty,1}}+C\int_{0}^{T}\|u_{x}\|_{B^{0}_{\infty,1}}\|u_{x}\|_{B^{0}_{\infty,\infty,1}}+\|u^{2}\|_{B^{1}_{\infty,\infty}}\ud\tau\notag\\
\leq&\|u_{0x}\|_{B^{0}_{\infty,1}\cap B^{0}_{\infty,\infty,1}}+C\int_{0}^{T}\|u_{x}\|_{B^{0}_{\infty,1}}\|u_{x}\|_{B^{0}_{\infty,\infty,1}}+\|u^{2}\|_{\mathcal{C}^{0,1}}\ud\tau\notag\\
\leq&CN^{\frac{9}{10}}+\underbrace{CN^{-\frac{1}{2}}\cdot\ln N}_{\text{which is a small quantity}}\|u_{x}\|_{L^{\infty}_{T}(B^{0}_{\infty,\infty,1})}+C\notag\\
\leq&CN^{\frac{9}{10}}.\label{uxinin1}
\end{align}
Plugging \eqref{uxinin1} into \eqref{bern}, we discover
\begin{align}\label{eb1}
\sum\limits_{j}2^{j}\big\|I_{3}\|_{L^{\infty}}\leq CN^{\frac{9}{10}-\frac{1}{2}}(\ln N)^{2}+CN^{-\frac{1}{2}-\frac{1}{5}}\ln N.
\end{align}
Multiplying both sides of \eqref{delta} by $2^{j}$ and performing the $l^{1}$ summation, by ${\rm(i)}-{\rm(iii)}$ we gain for any $t\in[0,T]$
\begin{align*}
\|u(t)\|_{B^{1}_{\infty,1}}=&\sum\limits_{j}2^{j}\big\|\Delta_{j}u\|_{L^{\infty}}=\sum\limits_{j}2^{j}\big\|\Delta_{j}u\circ y\|_{L^{\infty}}\\
\geq&t\|E_{0}\|_{B^{1}_{\infty,1}}-Ct\big(N^{-\frac{1}{10}}\ln N-N^{\frac{9}{10}-\frac{1}{2}}(\ln N)^{2}-N^{-\frac{1}{2}-\frac{1}{5}}\ln N\big)-\|u_{0}\|_{B^{1}_{\infty,1}}\\
\geq&Ct\Big(\frac{1}{4}N^{\frac{3}{5}}-N^{-\frac{1}{10}}\ln N-N^{\frac{9}{10}-\frac{1}{2}}(\ln N)^{2}-N^{-\frac{1}{2}-\frac{1}{5}}\ln N\Big)-C\\
\geq&\frac{1}{8}tN^{\frac{3}{5}}-C.
\end{align*}
where the second inequality holds by \eqref{fanli}. That is
$$\|u(t)\|_{B^{1}_{\infty,1}}\geq\frac{1}{16}N^{\frac{3}{5}-\frac{1}{2}}-C,\quad \forall  t\in [\frac{1}{2N^{\frac{1}{2}}},\frac{1}{N^{\frac{1}{2}}}].$$
Hence,
\begin{align}
\sup\limits_{t\in[0,\frac{1}{N^{\frac{1}{2}}}]}\|u(t)\|_{B^{1}_{\infty,1}}\geq\frac{1}{16}N^{\frac{3}{5}-\frac{1}{2}}-C>\ln N
\end{align}
which contradicts the hypothesis  \eqref{uxln}.

In conclusion, we obtain for $N>10$ large enough
\begin{align*}
&\|u\|_{L^{\infty}_{\bar{T}}(B^{1}_{\infty,1})}\geq\|u_{x}\|_{L^{\infty}_{\bar{T}}(B^{0}_{\infty,1})}\geq\ln N,\quad\quad \bar{T}=\frac{2}{N^{\frac{1}{2}}},\\
&\|u_{0}\|_{B^{1}_{\infty,1}}\lesssim N^{-\frac{1}{10}},
\end{align*}
that is we get the norm inflation and hence the ill-posedness of the CH equation. Thus, Theorem \ref{ill} is proved.
\end{proof}
\begin{rema}
In the proof, we have constructed an initial data $u_{0}$ such that
$$\|u_{0x}\|_{B^{0}_{\infty,1}}\lesssim N^{-\frac{1}{10}},~~\|u^2_{0x}\|_{B^{0}_{\infty,1}}\approx N^{\frac{7}{10}}.$$
Since $\|f^2 \|_{B^{0}_{\infty,1}\cap B^{0}_{\infty,\infty,1}}\leq \|f\|_{B^{0}_{\infty,1}}\|f \|_{B^{0}_{\infty,\infty,1}}$, to obtain the norm inflation in $B^{1}_{\infty,1}(\mathbb{R})$, one needs to construct an initial data $u_{0}$ satisfying the following approximate saturated inequality:
\begin{align}\label{guanjian}
c_2N^{\frac{7}{10}}\approx\|u^2_{0x}\|_{B^{0}_{\infty,1}}\leq\|u^2_{0x}\|_{B^{0}_{\infty,1}\cap B^{0}_{\infty,\infty,1}}\approx c_1N^{\frac{8}{10}},~~c_1\geq c_2.
\end{align}
Note that the CH equation is local well-posedness in $B^{1}_{\infty,1}\cap B^{1}_{\infty,\infty,1}$ (see Theorem \ref{well}), so if we want to prove the ill-posedness in $B^{1}_{\infty,1}(\mathbb{R})$, the equality \eqref{guanjian} will be the key skill.
\end{rema}

\subsection{A subset without norm inflation}
\par
In this subsection, we give the local existence and uniqueness for \eqref{ch} with initial data $u_{0}$ in $A:=\{f\in B^{1}_{\infty,1}|~f_{x}^{2}\in B^{0}_{\infty,1}\}$, a subset in $B^{1}_{\infty,1}$.
\begin{proof}[\rm\textbf{The proof of Theorem \ref{non}}:] Thanks to $u_0\in{B^1_{\infty,1}}\hookrightarrow \mathcal{C}^{0,1}$, one can obtain a unique solution $u\in L^{\infty}_T(\mathcal{C}^{0,1})$ ($\|u\|_{L^{\infty}_T(\mathcal{C}^{0,1})}\leq C\|u_0\|_{\mathcal{C}^{0,1}}$), see Theorem \ref{c01ill}. Moreover, one can choose a time $0< T_{0}<1$ sufficiently small such that $\frac{1}{2}\leq y_{\xi}(t,\xi)\leq 2$ where $y(t,\xi)$ satisfies \eqref{char}.

Noting that $E(t,x)=-(1-\partial_{xx})^{-1}\partial_{x}(\frac{u^2_x}{2})$, we rewrite the CH equation as follows:
\begin{equation*}
\left\{\begin{array}{l}
\frac{\ud}{\ud t}u+u\partial_{x}u=E-(1-\partial_{xx})^{-1}\partial_{x}(u^2),\quad t\in (0,T_{0}],\\
u(0,x)=u_0(x).
\end{array}\right.
\end{equation*}
Using again the embedding $\mathcal{C}^{0,1}\hookrightarrow B^{1}_{\infty,\infty}$, we see
\begin{align}
\|(1-\partial_{xx})^{-1}\partial_{x}(u^2)\|_{B^{1}_{\infty,1}}\leq  C\|u^2\|_{\mathcal{C}^{0,1}}\leq  C\|u_0\|^2_{B^1_{\infty,1}}.\label{u2}
\end{align}
Therefore, according to Lemma \ref{priori estimate} and the Gronwall inequality, we obtain
\begin{align}
\|u(t)\|_{B^1_{\infty,1}}\leq C \Big(\|u_0\|_{B^1_{\infty,1}}+\|u_0\|^2_{B^1_{\infty,1}}+ \int_0^t\|E(s)\|_{B^1_{\infty,1}}\ud s \Big),\qquad\ \forall~t\in(0,T_{0}].\label{ub1}
\end{align}
On the other hand, according to Lemma \ref{priori estimate}, we see from \eqref{E}--\eqref{nonlin} that
\begin{align}
\|E(t)\|_{B^{1}_{\infty,1}}\leq
C\Big(\|E_0\|_{B^{1}_{\infty,1}}+ \int_0^t \|u(s)\|_{B^{1}_{\infty,1}}\|E(s)\|_{B^{1}_{\infty,1}}+ \|u(s)\|^3_{B^{1}_{\infty,1}}{\ud}s\Big),\quad\ \forall~t\in(0,T_{0}].\label{ee}
\end{align}
Finally, combining \eqref{ub1} and \eqref{ee}, we deduce the following priori estimate
\begin{align}
\|u(t)\|_{B^{1}_{\infty,1}}+\|E(t)\|_{B^{1}_{\infty,1}}
\leq & C \Big(\|u_{0}\|_{B^{1}_{\infty,1}}+\|u_{0}\|^2_{B^{1}_{\infty,1}}+\|E_{0}\|_{B^{1}_{\infty,1}}
\notag\\
&+\int_0^t (\|u(s)\|^2_{B^{1}_{\infty,1}}+1)(\|E(s)\|_{B^{1}_{\infty,1}}+ \|u(s)\|_{B^{1}_{\infty,1}}){\ud}s\Big),\quad \forall~t\in[0,T_0].
\end{align}
Therefore, taking the similar calculations in \cite{yyg}, we can choose a lifespan $T\approx \frac{1}{C(\|u_{0}\|^2_{B^{1}_{\infty,1}}+\|E_{0}\|_{B^{1}_{\infty,1}}+1)}$ such that $\|u\|_{{L}^{\infty}_T(B^{1}_{\infty,1})}\leq C\|u_0\|_{B^{1}_{\infty,1}}$. Then, one can obtain a unique solution $u(t,x)\in \mathcal{C}_T(B^{1}_{\infty,1})\cap  \mathcal{C}^{1}_T(B^{0}_{\infty,1})$ and the norm inflation will not occur at this time. This proves Theorem \ref{non}.
\end{proof}

\noindent\textbf{Acknowledgements.}
The authors are thankful to Professor Pierre-Gilles, Lemari\'{e}-Rieusset and Professor Changxing, Miao for helpful discussion in constructing a counter example to present $B^{0}_{\infty,1}$ is not a Banach algebra. Guo was partially supported by GuangDong Basic and Applied Basic Research Foundation (No. 2020A1515111092) and Research Fund of Guangdong-Hong Kong-Macao Joint Laboratory for Intelligent Micro-Nano Optoelectronic Technology (No. 2020B1212030010). Ye and Yin were partially supported by NNSFC (Nos. 12171493 and 11671407),  FDCT (No. 0091/2018/A3), the Guangdong Special Support Program (No. 8-2015), and the key project of NSF of Guangdong province (No. 2016A03031104).

\addcontentsline{toc}{section}{\refname}

\end{document}